\documentclass[12pt,oneside,reqno]{amsart}
\usepackage{latexsym,amsmath,amsfonts,amsthm}
\usepackage[usenames]{color}
\usepackage{epsfig}
\usepackage{upgreek}

\usepackage{mathptmx}
\usepackage{bookman}
\usepackage{relsize} % useful for commands like \mathlarger for large math

\usepackage[]{graphicx}
\usepackage[font=small,labelfont=bf]{caption}
\usepackage{sidecap}
\usepackage{float}

\input xy
\xyoption{all}

\usepackage{hyperref}

\usepackage{amssymb}
\usepackage{amscd}
\usepackage{array}
\usepackage{verbatim}
\theoremstyle{definition}
\newtheorem{thm}{Theorem}[section]
\newtheorem{lem}[thm]{Lemma}
\newtheorem{prp}[thm]{Proposition}
\newtheorem{dfn}[thm]{Definition}
\newtheorem{cor}[thm]{Corollary}

\newtheorem{rmk}[thm]{Remark}

\def\vartau{\uptau}

\def\T{\bold T}

\def\ccite#1{\textcolor{Red}{\cite{#1}}}

\numberwithin{equation}{section}

\begin{document}

\bigskip

\title[S\lowercase{emiflat} O\lowercase{rbifold} P\lowercase{rojections}]
{\LARGE \rm S\lowercase{emiflat} O\lowercase{rbifold} P\lowercase{rojections}}

\author{S. W\lowercase{alters}}
\date{Oct. 19, 2017}
\dedicatory{In Memory of Mom}
\address{Department of Mathematics and Statistics, University of Northern B.C., Prince George, B.C. V2N 4Z9, Canada.}
\email[]{samuel.walters@unbc.ca \ \text{or} \ walters@unbc.ca}
\urladdr{hilbert.unbc.ca \ or \ web.unbc.ca/~walters}

\begin{abstract}
We compute the semiflat positive cone $K_0^{+SF}(A_\theta^\sigma)$  of the $K_0$-group of the irrational rotation orbifold $A_\theta^\sigma$ under the noncommutative Fourier transform $\sigma$ and show that it is determined by  classes of positive trace and the vanishing of two topological invariants. The semiflat orbifold projections are 3-dimensional and come in three basic topological genera: $(2,0,0)$, $(1,1,2)$, $(0,0,2)$. (A projection is called semiflat when it has the form $h + \sigma(h)$ where $h$ is a flip-invariant projection such that $h\sigma(h)=0$.) Among other things, we also show that every number in $(0,1) \cap (2\mathbb Z + 2\mathbb Z\theta)$ is the trace of a semiflat projection in $A_\theta$. The noncommutative Fourier transform is the order 4 automorphism $\sigma: V \to U \to V^{-1}$ (and the flip is $\sigma^2$: $U \to U^{-1},\ V \to V^{-1}$), where $U,V$ are the canonical unitary generators of the rotation algebra $A_\theta$ satisfying $VU = e^{2\pi i\theta} UV$.
\end{abstract}

\maketitle

\begin{quote}
{\Small \tableofcontents}
\end{quote}

\newpage

%%%%%%%%%%%%%%%%%%%%%%%%%%%%%%%%%%%%%%%%%
\textcolor{blue}{\large\section{Introduction}}

For each irrational number $\theta$ in $(0,1)$ the irrational rotation C*-algebra $A_\theta$ is the universal (and unique) C*-algebra generated by unitaries $U,V$ satisfying the  Heisenberg relation
\[
VU = e^{2\pi i\theta} UV.
\]
The noncommutative Fourier transform is the canonical order four automorphism $\sigma$ of $A_\theta$ defined by the equations
\[
\sigma(U) = V^{-1}, \quad \sigma(V) = U.
\]
The flip automorphism $\Phi = \sigma^2$ of the rotation C*-algebra $A_\theta$ is defined by $\Phi(U) = U^{-1}, \Phi(V) = V^{-1}$, and it was studied extensively in \ccite{BEEKa} \ccite{BEEKb} \ccite{BK} \ccite{SWprojmodules} \ccite{SWcmp}.

If one represents the unitaries $U, V$ on the Hilbert space $L^2(\mathbb R)$ in the canonical way as complex phase multiplication and translation operators, the automorphism $\sigma$ corresponds exactly to the classical Fourier transform on $L^2$, hence the name.

In this paper we show that the semiflat positive cone $K_0^{+SF}(A_\theta^\sigma)$  of the $K_0$-group of the irrational rotation orbifold $A_\theta^\sigma$ (a sort of noncommutative sphere \ccite{SWcjm}) is determined by classes of positive trace and the vanishing of two topological invariants (Theorem \ref{maintheorem}). We also determine semiflat and flat projections by their canonical traces and topological invariants up to Fourier-invariant unitary equivalence (Theorems \ref{identifyflat} and \ref{identifysemiflat}); namely, Fourier invariant projections as well as projections that are orthogonal to their transform (of which there are two kinds). First, let us define these.

\medskip

\begin{dfn}
By a {\it cylic} (or {\it $\sigma$-cyclic}) projection in $A_\theta$ we mean a projection $g$ that is orthogonal to its $\sigma$-orbit -- that is, $g, \sigma(g), \sigma^2(g), \sigma^3(g)$ are mutually orthogonal. The associated $\sigma$-invariant projection
\begin{equation}\label{defofflat}
f = \sigma^*(g) := g + \sigma(g) + \sigma^2(g) + \sigma^3(g)
\end{equation}
will be called {\it flat} (or {\it $\sigma$-flat}). We refer to $g$ as a {\it cyclic projection} for $f$. 
\end{dfn}

Flat projections have been used in \ccite{SWK-ind} to obtain the K-inductive structure of the Fourier transform $\sigma$.

\begin{dfn}
A projection $h$ is {\it semicyclic} (with respect to $\sigma$) if $h\sigma(h) = 0$ and $\sigma^2(h) = h$ (i.e., $h$ is flip invariant). The associated Fourier invariant projection $f = h + \sigma(h)$ is called {\it semiflat}. Two projections are {\it $\sigma$-unitarily equivalent} when they are unitarily equivalent by a unitary that is $\sigma$-invariant.
\end{dfn}

The sum of two orthogonal flat (resp., semiflat) projections is flat (resp., semiflat). If $g$ is a cyclic projection, then $g + \sigma^2(g)$ is semicyclic. 

We will see that the trace and the topological genus of semiflat projections determines them up to $\sigma$-unitary equivalence. (The meaning of topological genus is given in Definition \ref{toptype}.)

\begin{dfn}
The {\it semiflat positive cone}, denoted $K_0^{+SF}(A_\theta^\sigma)$, consists of the nonzero classes in positive cone $K_0^+(A_\theta^\sigma)$ given by semiflat projections in $A_\theta^\sigma$. 
\end{dfn}

(We excluded the zero class for simplicity, although we could have included it.)

\begin{thm}\label{maintheorem} (Main Theorem.)
{\it Let $\theta$ be irrational. The semiflat positive cone $K_0^{+SF}(A_\theta^\sigma)$ consists of the $K_0$-classes $x$ of positive trace $\vartau(x)$ and
\[
\psi_{10}(x) = \psi_{11}(x) = 0.
\]
Further, the semiflat projections come in three basic topological genera: $(2,0,0)$, $(1,1,2)$, $(0,0,2)$.
}
\end{thm}

Thus, in general, a semiflat projection has genus that is an integral linear combination of these three basic genus types. 

\begin{cor}
{\it Let $\theta$ be irrational and let $f, h$ be two semiflat projections in $A_\theta^\sigma$. Then $f$ and $h$ are $\sigma$-unitarily equivalent if and only if they have the same trace and same genus.
}
\end{cor}

It is well known that the $K_0$-group of $A_\theta$ is $\mathbb Z^2$, and that the unique tracial state $\vartau$ of $A_\theta$ induces a group isomorphism $\vartau_*: K_0(A_\theta) \to \mathbb Z + \mathbb Z\theta$ when $\theta$ is irrational. (This is a classic theorem of Pimsner and Voiculescu \ccite{PV}, and Rieffel \ccite{MR1981} from 1980-81.) Further, it is also known that each number in $(0,1) \cap (\mathbb Z + \mathbb Z\theta)$ is the trace of a projection in $A_\theta$, namely a Powers-Rieffel projection \ccite{MR1981}.  For our purposes here, we prove the following related results for cyclic, flat, and semiflat projections.

\begin{thm}
{\it Let $\theta$ be irrational.  Each number in $(0,\tfrac14) \cap (\mathbb Z + \mathbb Z\theta)$ is the trace of a cyclic projection in $A_\theta$.
}
\end{thm}

\begin{thm}\label{flattraces}
{\it Let $\theta$ be irrational. Each number in $(0,1) \cap (4\mathbb Z + 4\mathbb Z\theta)$ is the trace of a flat projection in $A_\theta$.
}
\end{thm}

The analogous results hold for semicyclic and semiflat projections.

\begin{thm}
{\it Let $\theta$ be irrational.  Each number in $(0,\tfrac12) \cap (\mathbb Z + \mathbb Z\theta)$ is the trace of a semicyclic projection in $A_\theta$.
}
\end{thm}

\begin{thm}\label{semiflattraces}
{\it Let $\theta$ be irrational.  Each number in $(0,1) \cap (2\mathbb Z + 2\mathbb Z\theta)$ is the trace of a semiflat projection in $A_\theta$.
}
\end{thm}

The next result shows that the vanishing of all the topological invariants of a Fourier invariant projection means that it must be flat.

\begin{thm}\label{zerotopflat}
{\it Let $e$ be a Fourier invariant projection in $A_\theta$ where $\theta$ is irrational.  
If the topological invariants of $e$ vanish (i.e., $\psi_{**}(e)=0$), then $e$ is flat. 
}    
\end{thm}
\begin{proof}
By Lemma \ref{traces2and4} the trace of $e$ has to be in $4\mathbb Z + 4\mathbb Z\theta$. By Theorem \ref{flattraces}, $\vartau(e)$ would also be the trace of some flat projection $f$. Since the Connes-Chern character invariant $\T_4$ (mentioned below) of $e$ and $f$ are equal, they are unitarily equivalent by a $\sigma$-invariant unitary, and hence $e$ is flat also.
\end{proof}

This characterizes what one might call the flat positive cone $K_0^{+F}(A_\theta^\sigma)$, the nonzero classes in the positive cone $K_0^+(A_\theta^\sigma)$ represented by flat projections in $A_\theta^\sigma$.

We also obtain a result that identifies cyclic and flat projections up to Fourier invariant unitary equivalence simply by means of the trace.

\begin{thm}\label{identifyflat}
{\it Let $\theta$ be any irrational number, and let $g_1$ and $g_2$ be two cylic projections in $A_\theta$.
\begin{enumerate}
\item Then $g_1$ and $g_2$ are $\sigma$-unitarily equivalent iff $g_1$ and $g_2$ have the same trace.
\item Two flat projections $f_1=\sigma^*(g_1)$ and $f_2 = \sigma^*(g_2)$ are $\sigma$-unitarily equivalent iff they have the same trace iff $g_1$ and $g_2$ $\sigma$-unitarily equivalent.
\end{enumerate}
}
\end{thm}

We also have the corresponding result for semiflat projections and their semicyclic components.

\begin{thm}\label{identifysemiflat}
{\it Let $\theta$ be any irrational number, and let $g$ and $h$ be two semicyclic projections in $A_\theta^\Phi$.
\begin{enumerate}
\item Then $g$ and $h$ are $\sigma$-unitarily equivalent iff they are $\Phi$-unitarily equivalent iff  $\T_2(g) = \T_2(h)$.
\item Two semiflat projections $f=g+\sigma(g)$ and $f' = h + \sigma(h)$ are $\sigma$-unitarily equivalent iff $\vartau(g) = \vartau(h)$ and
\[
\phi_{00}(g) = \phi_{00}(h), \qquad \phi_{11}(g) = \phi_{11}(h), \qquad
\phi_{01}(g) + \phi_{10}(g) = \phi_{01}(h) + \phi_{10}(h).
\]
\end{enumerate}
}
\end{thm}

\medskip

Lastly, we show that all possible trace values are realized by Fourier invariant projections.

\begin{thm} (See Theorem \ref{TracesFourierInvariant}.)
{\it Let $\theta$ be irrational. Each number in $(0,1) \cap (\mathbb Z + \mathbb Z\theta)$ is the trace of a Fourier invariant projection in $A_\theta$.
}
\end{thm}

%%%%%%%%%%%%%%%%%%%%%%%%%%%%%%%%%%%%%%%%%
\textcolor{blue}{\large \section{Topological Invariants}}

In this section we recall the topological invariants for the flip and the Fourier transform associated with the ``twisted" unbounded traces on the canonical smooth dense *-subalgebra $A_\theta^\infty$.

\medskip

The flip automorphism $\Phi$ has associated unbounded $\Phi$-traces defined on the basic unitaries $U^mV^n$ by 
\begin{equation}
\phi_{ij}(U^mV^n)\ =\ e(-\tfrac{\theta}2 mn)\,\updelta_2^{m-i} \updelta_2^{n-j}
\end{equation}
for $ij = 00, 01,10,11$,  $m,n\in \mathbb Z$, where $\updelta_a^{b}$ is the divisor delta function defined to be 1 when $a$ divides $b$, and 0 otherwise.
(See \ccite{SWprojmodules} or \ccite{SWcmp}.) These are (unbounded) linear functionals defined on the canonical smooth dense *-subalgebra $A_\theta^\infty$ which are $\Phi$-invariant and satisfy the $\Phi$-trace condition
\[
\phi_{ij}(xy) = \phi_{ij}(\Phi(y)x)
\]
for all $x,y$ in $A_\theta^\infty$. In addition, they are Hermitian maps: they are real on Hermitian elements.  Clearly, on the fixed point subalgebra $A_\theta^{\infty,\Phi}$ of $A_\theta^\infty$ under the flip they give rise to (unbounded) trace functionals. Together with the canonical trace $\vartau$ one has the Connes-Chern character  
\begin{equation} 
\T_2 : K_0(A_\theta^\Phi) \to \mathbb R^5, \qquad
\T_2(x) = (\vartau(x); \phi_{00}(x), \phi_{01}(x), \phi_{10}(x), \phi_{11}(x))
\end{equation}
the injectivity of which was shown in \ccite{SWprojmodules} (Proposition 3.2) for irrational $\theta$.\footnote{In \ccite{SWprojmodules} we worked with the crossed product algebra $A_\theta \rtimes_\Phi \mathbb Z_2$, but since this is strongly Morita equivalent to the fixed point algebra, the injectivity follows.} We may sometimes refer to $\T_2(x)$, or simply the $\phi_{ij}(x)$, as the $\Phi$-topological invariant(s) of the class. For the identity element one has $\T_2(1) = (1; 1, 0, 0, 0)$.

\medskip

For the Fourier transform $\sigma$, one has five basic unbounded twisted trace functionals defined on the canonical smooth *-subalgebra $A_\theta^\infty$ of $A_\theta$, defined as follows on generic unitary elements:
\begin{align}\label{unboundedtraces}
\psi_{10}(U^mV^n) &= e(-\tfrac\theta4(m+n)^2)\,\updelta^{m-n}_2, &
\psi_{20}(U^mV^n) &= e(-\tfrac\theta2 mn)\,\updelta^m_2\updelta^n_2, 
 \\ 
\psi_{11}(U^mV^n) &= e(-\tfrac\theta4(m+n)^2)\,\updelta^{m-n-1}_2, &
\psi_{21}(U^mV^n) &= e(-\tfrac\theta2 mn) \,\updelta^{m-1}_2 \updelta^{n-1}_2,
 \\
& \ \ & 
\psi_{22}(U^mV^n) &= e(-\tfrac\theta2 mn)\,\updelta^{m-n-1}_2.
\end{align}
(See \ccite{SWChern}\footnote{In \ccite{SWChern} our Fourier transform was the inverse of the one used in this paper, so the unbounded traces in \ccite{SWChern} are ``conjugate" to those above. See also the proof of Lemma \ref{lemma9vectors}.}.) These maps were calculated in \ccite{SWChern} and were used in \ccite{SWcjm}, \ccite{SWcrelles}. The Fourier Connes-Chern character is the group homomorphism
\[
\T_4: K_0(A_\theta^\sigma) \to \mathbb C^6, \qquad
\T_4(x) = (\vartau(x); \ \psi_{10}(x), \psi_{11}(x);\ \psi_{20}(x), \psi_{21}(x), \psi_{22}(x))
\]
where $\vartau$ is the canonical trace on $A_\theta^\sigma$, the (orbifold) fixed point subalgebra with respect to the Fourier transform. For irrational $\theta$ the map $\T_4$ is injective, so defines a complete invariant for projections in the fixed point algebra $A_\theta^\sigma$ (up to Fourier invariant unitary equivalence).\footnote{When $\theta$ is rational, one needs to include the Chern number arising from Connes' cyclic 2-cocycle to the $\T_4$ invariant to ensure injectivity -- however, for our purposes, this is not necessary.} For the identity one has $\T_4(1) = (1; 1,0; 1, 0, 0)$.

\medskip

\begin{dfn}\label{toptype}
By the {\it topological genus} (or simply {\it genus}) of a semiflat projection $f$ we mean the triple $(\psi_{20}(f), \psi_{21}(f), \psi_{22}(f))$.
\end{dfn}

The injectivity of the Fourier Connes-Chern map $\T_4$ was shown in  \ccite{SWcjm} for a dense $G_\updelta$ set of $\theta$'s, but later it was shown in \ccite{AP}, and independently in \ccite{ELPW}, that $K_0(A_\theta^\sigma) \cong \mathbb Z^9$ for all $\theta$, which gives the injectivity of $\T_4$ for all irrational $\theta$.  This allows us to conclude that since $A_\theta^\sigma$ has the cancellation property for any irrational $\theta$, two projections $e$ and $e'$ in $A_\theta^\sigma$ are $\sigma$-unitarily equivalent if and only if $\T_4(e)=\T_4(e')$. 

One easily checks the following relations between the $\Phi$ and $\sigma$ unbounded traces
\begin{equation}\label{phiandpsi}
\psi_{20} = \phi_{00}, \qquad \psi_{21} = \phi_{11}, \qquad
\psi_{22} = \phi_{01} + \phi_{10}.
\end{equation} 
We will need to use the parity automorphism $\gamma$ of $A_\theta$ defined by
\[
\gamma(U) = -U, \qquad \gamma(V) = -V
\]
which will be useful because it commutes with the Fourier transform and has the property of switching the signs of the topological maps $\psi_{11},\psi_{22}$ (while preserving the others). It also has the useful property 
\begin{equation}\label{phigamma}
\phi_{00}\gamma = \phi_{00}, \qquad \phi_{11}\gamma = \phi_{11}, \qquad
\phi_{01}\gamma = -\phi_{01}, \qquad \phi_{10}\gamma = -\phi_{10}.
\end{equation}

\medskip

The topological numbers of $\sigma$-invariant projections are quantized. Indeed, in view of \ccite{SWChern} and \ccite{SWcjm}, the $\psi_{10}, \psi_{11}$ invariants of such projections take values in the lattice subgroup $\mathbb Z + \mathbb Z(\frac{1-i}{2})$ of $\mathbb C$; the $\psi_{20}, \psi_{21}$ invariants take values in $\frac{1}{2}\mathbb Z$, and $\psi_{22}$ in $\mathbb Z$. 

\medskip

Analogous results can probably be established for the Cubic and Hexic transforms studied in \ccite{BW} and \ccite{SWnuclearphysicsb}. For example, for the Hexic transform $\rho$ (the canonical order 6 automorphism), there are three kinds of `flat' projections: $g+\rho(g)$ (where $g$ is $\rho^2$-invariant), $g+\rho(g)+\rho^2(g)$ (where $g$ is $\rho^3$-invariant, $\rho^3$ being the flip), and $g+\rho(g)+ \dots + \rho^5(g)$ (where $g$ is $\rho$-cyclic).

%%%%%%%%%%%%%%%%%%%%%%%%%%%%%%%%%%%%%%%%%
\textcolor{blue}{\large \section{Density of Topological Types}}

In this section we establish key lemmas needed for the proof of the main theorem. 
For the reader's convenience, we quote the part of Lemma  3.1 from \ccite{SWprojmodules} (p. 594) that is relevant to the proofs below.

\medskip

\begin{lem}\label{lemmaSWprojmodules}
{\it Let $\alpha = r\theta+s$ be irrational in the interval $(\frac12, 1)$ where $r,s$ are integers. With $U^r$ and $V$ being unitaries satisfying $VU^r= e^{2\pi i\alpha} U^rV$, there exists a Powers-Rieffel projection 
\[
e = Vg(U^r) + f(U^r) + g(U^r)V^{-1}
\]
of trace $\alpha$ that is flip-invariant, where $f,g$ are certain smooth functions. 
Further, if $r$ is even, then
\[
\phi_{ij}(f(U^r)) = 0, \qquad
\phi_{ij}(g(U^r)V^{-1}) = \begin{cases} 0 &\text{ if $s$ is even,} 
\\ \frac12 \updelta_2^{j-1}\updelta_2^{i}  &\text{ if $s$ is odd.} \end{cases} 
\]
If $r$ is odd, one has
\[
\phi_{ij}(f(U^r)) = \tfrac12(-1)^i \updelta_2^j, \qquad
\phi_{ij}(g(U^r)V^{-1}) = \tfrac14(-1)^{i(s+1)} \updelta_2^{j-1}.
\]
}
\end{lem}

(N.B., this slightly more simplified version of the lemma was obtained by setting ``$p=q=0$" in the notation of Lemma  3.1 of \ccite{SWprojmodules}.)

\medskip

\begin{lem}\label{ees}
{\it Let $\theta$ be irrational. There are flip-invariant Powers-Rieffel projections $e, e', e''$ in $A_\theta$ with $\Phi$-invariants
\[
\T_2(e) = (\vartau(e); 0,1,0,0), \qquad  
\T_2(e') = (\vartau(e'); \tfrac12,\tfrac12,\tfrac12,\tfrac12), \qquad \T_2(e'') = (\vartau(e''); 1,0,0,0)
\]
such that the set of traces of each type is dense in $(0,\frac12)$.
}
\end{lem}
\begin{proof}
Using Lemma \ref{lemmaSWprojmodules} with $\alpha = r\theta + s$ in $(\frac12,1)$ where $r$ is even and $s$ odd, the Powers-Rieffel projection
\[
e_1 = Vg(U^r) + f(U^r) + g(U^r)V^{-1}
\]
has invariant $\T_2(e_1) = (r\theta + s; 0,1,0,0)$. Indeed, in this case $\phi_{ij}(f(U^r)) = 0$ and 
\[
\phi_{ij}(e_1) = 2 \phi_{ij}(g(U^r)V^{-1}) = \updelta_2^{j-1}\updelta_2^{i}.  
\]
Thus, $\phi_{01}(e_1)=1$ and the other $\phi_{ij}(e_1)=0$.

Now let us take any other irrational $\alpha' = r'\theta + s'$ in $(\frac12,1)$ but this time with both $r'$ and $s'$ even. The corresponding Powers-Rieffel projection (Lemma \ref{lemmaSWprojmodules})
\[
e_2 = Vg(U^{r'}) + f(U^{r'}) + g(U^{r'})V^{-1}
\]
has $\phi_{ij}(e_2)=0$ and $\T_2(e_2) = (\alpha'; 0,0,0,0)$. Since the sets of such $\alpha$ and $\alpha'$ are dense in $(\frac12,1)$, choosing $1>\alpha >\alpha' > \frac12$ we get a dense set of traces $\{\alpha - \alpha'\}$ in $(0,\frac12)$. Upon picking a flip-invariant unitary $w$ such that $we_2w^* \le e_1$, we obtain the flip-invariant projection
\[
e = e_1 - we_2w^* 
\]
with invariant $\T_2(e) = (\vartau(e); 0,1,0,0)$ and traces dense in $(0,\frac12)$,  giving us the projections $e$ in the statement of the lemma.

Note that
\[
\T_2(1-e_2) = (1; 1,0,0,0) - (\alpha'; 0,0,0,0) = (1-\alpha'; 1,0,0,0)
\]
whose traces are dense in $(0,\frac12)$, which gives us the projections $e''$ in the statement of the lemma.

Now let's suppose that $r$ and $s$ are odd and let $e_{{}_0}$ be the associated Powers-Rieffel projection of trace $\beta = r\theta+s \in (\frac12,1)$. Then $\phi_{ij}(e_{{}_0}) = \tfrac12(-1)^i \updelta_2^j +  \tfrac12 \updelta_2^{j-1}$ and
\[
\T_2(e_{{}_0}) = (\beta; \tfrac12, \tfrac12, -\tfrac12, \tfrac12)
\]
which in view of \eqref{phigamma} gives
\[
\T_2(\gamma e_{{}_0}) = (\beta; \tfrac12, -\tfrac12, \tfrac12, \tfrac12).
\]
Subtracting from $\gamma e_{{}_0}$ subprojections equivalent to $e_2$'s of traces $\alpha'$ less than $\beta$ (as done previously), we obtain flip-invariant projections $e'''$ such that
\[
\T_2(e''') = \T_2(\gamma e_{{}_0}) - \T_2(e_2) = (\beta - \alpha'; \tfrac12, -\tfrac12, \tfrac12, \tfrac12)
\]
where the set of traces $\{\beta -\alpha'\}$ is dense in $(0,\frac12)$. Adding projections $\Phi$-unitarily equivalent to $e$ orthogonally to $e'''$ one gets flip-invariant projections $e'$ such that
\[
\T_2(e') = (\vartau(e'); \tfrac12, \tfrac12, \tfrac12, \tfrac12)
\]
with traces $\vartau(e')$ is dense in $(0,\frac12)$. 
\end{proof}

In view of Theorem \ref{semiflattraces}, the projections $e,e',e''$ in Lemma \ref{ees} can be conjugated by suitable flip-invariant unitaries $u$ so that, for instance, $h=ueu^*$ is under a semicyclic projection $g$ of some suitably larger trace. (Recall that this means $g\sigma(g)=0$ where $g$ is flip invariant.) Clearly, the $\T_2$ invariants of $e$ and $h$ are the same (since $u$ is flip-invariant), with the difference that $h$ is now a semicyclic projection. Therefore, we obtain the following.

\begin{cor}\label{corh}
{\it Let $\theta$ be irrational. There are semicyclic projections $h,h',h''$ in $A_\theta$ with invariants
\[
\T_2(h) = (\vartau(h); 0,1,0,0), \qquad 
\T_2(h') = (\vartau(h'); \tfrac12,\tfrac12,\tfrac12,\tfrac12), \qquad 
\T_2(h'') = (\vartau(h''); 1,0,0,0)
\]
such that the set of traces of each type is dense in $(0,\frac12)$.
}
\end{cor}

Now consider the semiflat projection $f = h + \sigma(h)$ associated to the first type of projection $h$ in this corollary. We have
\[
\psi_{20}(f) = 2 \psi_{20}(h) = 2\phi_{00}(h) = 0
\]
likewise $\psi_{21}(f) = 2 \phi_{11}(h) = 0$, and 
\[
\psi_{22}(f) = 2 \psi_{22}(h) = 2\phi_{01}(h) + 2\phi_{10}(h) = 2.
\]
Therefore $f$ is semiflat of topological genus $(0,0,2)$.\footnote{The author had some difficulty finding semiflat projections of genus $(0,0,2)$.} In addition, the traces of such $f$ form a dense set in $(0,1)$.

In the same way, from $h'$ and $h''$ we obtain semiflat projections $f'$ and $f''$ with respective topological genera $(1,1,2)$ and $(2,0,0)$.  We will say that a class of projections has {\it trace density} when the set of its traces is dense in $(0,1)$. We have therefore addressed part of the following result.

\begin{lem}\label{typeslemma}
{\it Let $\theta$ be irrational. Each triple 
\[
(1,1,2), \quad (2,0,0), \quad (0,0,2), \quad (0,2,0),
\]
\[
(-1,-1,-2), \quad (-2,0,0), \quad (0,0,-2),\quad (0,-2,0),
\]
is the topological genus of semiflat projections in $A_\theta^\sigma$ with dense traces in $(0,1)$.
}
\end{lem}

Applying the automorphism $\gamma$ to semiflats of genus $(0,0,2)$ we obtain 
semiflat projections of genus $(0,0,-2)$ with trace density.

Since $(0,2,0) = 2(1,1,2) + (-2,0,0) + 2(0,0,-2)$ is a positive linear combination of genera with trace density, we immediately get semiflat projections of genus $(0,2,0)$ with trace density.  To complete the proof Lemma \ref{typeslemma} we must deal with the remaining genera:
\[
(-2,0,0), \quad (-1,-1,-2), \quad (0,-2,0).
\]
These, however, follow from the lemma.

\begin{lem}\label{negativetype}
{\it For each semiflat projection of genus $(a,b,c)$, there is a semiflat projection of genus $(-a,-b,-c)$. If the former has trace density, so does the latter.
}
\end{lem}
\begin{proof}
Let $p = h+\sigma(h)$ be a semiflat projection of genus $(a,b,c)$ (with trace density). We show how to construct semiflat projections of its negative genus $(-a,-b,-c)$ (with trace density). By Theorem \ref{flattraces}, choose a flat projection
\[
k = g + \sigma(g) + \sigma^2(g) + \sigma^3(g)
\]
where $g$ is a cyclic projection such that 
$2\vartau(h) = \vartau(p) < \vartau(k) = 4\vartau(g)$. Choose a flip-invariant unitary $u$ such that $uhu^* \le g+ \sigma^2(g)$, where $g+ \sigma^2(g)$ is flip-invariant and semicyclic. The difference projection
\[
\tilde h := g + \sigma^2(g) - uhu^*
\]
is semicyclic as well (and is flip invariant) and its associated semiflat projection $\tilde p := \tilde h + \sigma(\tilde h)$ has genus $(-a,-b,-c)$ (since the $\psi_{ij}$ invariants of $g$ and $k$ vanish). Further, the traces
\[
\vartau(\tilde p) = 2\vartau(\tilde h) = 4 \vartau(g) - 2\vartau(h) = \vartau(k) - \vartau(p)
\]
are dense in $(0,1)$ (since the set of traces $\vartau(k)$ and $\vartau(p)$ are each dense in $(0,1)$).
\end{proof}

This completes the proof of Lemma \ref{typeslemma}.

\begin{lem}\label{traces2and4}
{\it Let $x$ be a class in $K_0(A_\theta^\sigma)$. If $\psi_{10}(x)=\psi_{11}(x)=0$, then $\vartau(x) \in 2\mathbb Z + 2\mathbb Z\theta$.  If all $\psi_{ij}(x)=0$, then $\vartau(x) \in 4\mathbb Z + 4\mathbb Z\theta$. 
}
\end{lem}

The proof of this lemma is contained in the proof of Theorem \ref{maintheorem} given in Section 4 below (see paragraph following equation \eqref{B} below).

%%%%%%%%%%%%%%%%%%%%%%%%%%%%%%%%%%%%%%%%%
%\section{\textcolor{blue}{\large Topological Structure of Semiflat Projections}}
\textcolor{blue}{\large \section{Proof of Main Theorem}}

In this section we prove the main theorem on the determination of the semiflat positive cone of $K_0$ of the Fourier orbifold $A_\theta^\sigma$. Before doing so we state a lemma that is essentially a paraphrase of a result from \ccite{SWChern}, which was originally stated for crossed products, for our fixed point subalgebra situation.
 
 \medskip
 
\begin{lem}\label{lemma9vectors}
{\it The range of the homomorphism $\T_4: K_0(A_\theta^\sigma) \to \mathbb C^6$ is spanned by the nine vectors
\begin{align*}
V_1 &= (2; \ 0, 0; \ 2, 0, 0) \\
V_2 &= (2;\ 1+i, 0;\ 0, 0, 0) \\
V_3 &= (1;\ 1, 0;\ 1, 0, 0) \\
V_4 &= (2;\ 0, 0;\ 0, 2, 0) \\
V_5 &= (2;\ 0, 1+i;\ 0, 0, 0) \\
V_6 &= (1;\ 0, 1;\ 0, 1, 0) \\
V_7 &= (\theta;\ \tfrac12 - \tfrac12 i,\ \tfrac12 - \tfrac12 i; \ \ \tfrac12, \tfrac12, 1) \\  
V_8 &= (\theta;\ -\tfrac12 - \tfrac12 i,\ -\tfrac12 - \tfrac12 i; \ \ -\tfrac12, -\tfrac12, -1) \\
V_9 &= (\theta;\ -\tfrac12 + \tfrac12 i,\ -\tfrac12 + \tfrac12 i;\ \ \tfrac12, \tfrac12, 1).   
\end{align*}
}
\end{lem}
\begin{proof}
These can be obtained from the range of the associated homomorphism calculated for the crossed product $A_\theta \rtimes_\sigma \mathbb Z_4$ in \ccite{SWChern} (see character table on page 645). The only difference that we need to take into account is that the ``Fourier transform" used in \ccite{SWChern} was the inverse of the one used in the current paper -- and this has the effect of taking the complex conjugates of the $\psi_{10}, \psi_{11}$ values obtained in \ccite{SWChern}, which correspond, respectively, to the values of the maps ``$T_{10}$" and ``$\lambda^{1/4}T_{11}$" used in the character table therein. Further, we need to multiply all entries in that character table by 4 in view of the normalizations used for the unbounded traces in \ccite{SWChern}. Once these are taken into account, we obtain the above 9 vectors from those in \ccite{SWChern} in view of the canonical isomorphism $K_0(A_\theta^\sigma) \cong K_0(A_\theta \rtimes_\sigma \mathbb Z_4)$.
\end{proof}

\medskip

We are now ready to prove the Main Theorem \ref{maintheorem}.

\begin{proof} (Proof of Theorem \ref{maintheorem}.)
Fix a class $x$ in $K_0(A_\theta^\sigma)$ such that $\vartau(x) > 0$ and $\psi_{10}(x) = \psi_{11}(x)=0$. If $\vartau(x) > 1$, by Theorem \ref{semiflattraces} we can subtract from $x$ the sum of a finite number of $K_0$-classes of semiflat projections so that the difference has positive trace less than 1. Further, the fact that $\vartau(x)\not= 1$ will follow from the computation below which show that the vanishing of $\psi_{10}(x)$ and $\psi_{11}(x)$ implies that the trace of $x$ is a multiple of 2 -- see, for example, equation \eqref{startrace} below. Therefore, with no loss of generality we may assume that $0 < \vartau(x) < 1$.

Write $\T_4(x)$ as an integral linear combination of the nine vectors in Lemma \ref{lemma9vectors}:
\[
\T_4(x) = \sum_{j=1}^9 N_j V_j
\]
for some integers $N_j$. Reading off the $\psi_{10}$ and $\psi_{11}$ coordinates of $x$, we have
\[
\psi_{10}(x) = N_2(1+i) + N_3 + N_7(\tfrac12 - \tfrac12 i) + N_8(-\tfrac12 - \tfrac12 i)
+ N_9(-\tfrac12 + \tfrac12 i) \ = \ 0 
\]
and
\[
\psi_{11}(x) = N_5(1+i) + N_6 +  N_7(\tfrac12 - \tfrac12 i) + N_8(-\tfrac12 - \tfrac12 i)
+ N_9(-\tfrac12 + \tfrac12 i) \ = \ 0. 
\]
Solving these gives
\[
N_6 = N_3, \qquad N_5 = N_2, \qquad N_7 = N_9 - N_3, \qquad 
N_8 = 2N_2 + N_3.
\]
In terms of the integers $N_1, N_2, N_3, N_4, N_9$, the total trace is
\begin{equation}\label{startrace}
\vartau(x) = \sum_j N_j\vartau(V_j) = 2N_1 + 4N_2 + 2N_3 + 2N_4 + ( 2N_2 + 2N_9)\theta
\end{equation}
which is a multiple of 2 as we had noted at the beginning of the proof. Simplifying, one gets the $\psi_{2k}$ invariants
\[
\psi_{20}(x) = 2N_1 - N_2 + N_9
\]
\[
\psi_{21}(x) = 2N_4 - N_2 + N_9
\]
\[
\psi_{22}(x) = 2N_9 - 2N_2 - 2N_3.
\]
Therefore one gets
\begin{align}
\T_4(x) 
&= (2N_1 + 4N_2 + 2N_3 + 2N_4 + (2N_2 + 2N_9)\theta;\ 
\ 0, 0; 	\label{A}	%\ \ \tag{$\star$} \label{A}
\\
&\qquad \qquad \quad 2N_1 - N_2 + N_9, \ \ 2N_4 - N_2 + N_9, \ \ 2N_9 - 2N_2 - 2N_3) \ \notag
\\
&= N_1 (2; \ 0, 0; \ 2, 0 , 0) + N_2 (4+2\theta; \ 0, 0; \ -1, -1, -2) 
+ N_3(2; \ 0, 0; \ 0, 0, -2)   \  \ \label{B}
\\
&\qquad \qquad \quad + N_4(2; \ 0, 0; \ 0, 2, 0) + N_9(2\theta; \ 0, 0; \ 1, 1, 2). \ \notag
\end{align}

We digress momentarily to note that in view of this calculation, the condition $\psi_{10}(x) = \psi_{11}(x)=0$ has yielded the conclusion that $\vartau(x)$ is in $2\mathbb Z + 2\mathbb Z\theta$ (as can be seen from \eqref{A}), thus establishing the first assertion of Lemma \ref{traces2and4}. If, in addition, the remaining invariants $\psi_{2k}(x)$ vanish, then it is easy to check that $\vartau(x)$ is in $4\mathbb Z + 4\mathbb Z\theta$, which establishes the second assertion of Lemma \ref{traces2and4}. Thus, in particular, if all the topological invariants of a $K_0$-class $x$ vanish, then its trace is a multiple of 4 - and we know from Theorem \ref{flattraces} that any such number is the trace of a flat projection.

\medskip

Returning to our current proof, in view of Lemma \ref{typeslemma} we can pick semiflat projections $f_1, f_2, f_3, f_4, f_9$ with the respective topological genera appearing in the last equality in \eqref{B}, and of appropriately small trace, such that 
\[
\T_4\left(x - \sum_{i=1,2,3,4,9} |N_i| [f_i]\right) = (\alpha; 0, 0; 0, 0, 0)
\]
where $\alpha < 1$ is some positive number in $\mathbb Z + \mathbb Z\theta$. We now have a class with all its topological invariants vanishing, so by Lemma \ref{traces2and4}, $\alpha = 4\alpha'$ for some $\alpha' \in \mathbb Z + \mathbb Z\theta$. By Theorem \ref{flattraces}, there is flat projection $f$ of trace $4\alpha'$ and $\T_4 [f] = (4\alpha'; 0, 0; 0, 0, 0)$. This gives
\[
\T_4(x) = \T_4\left([f] + \sum_{i=1,2,3,4,9} |N_i| [f_i]\right)
\]
and by the injectivity of $\T_4$, the class $x$ is a finite non-negative integral linear combination of classes of semiflat projections (at least one of them nonzero). Since $\vartau(x) < 1$, Lemma \ref{lemma1} allows us to write $x = [e]$ for a single semiflat projection $e$ (since the projections $f, f_i$ could all be unitarily combined into a sum of orthogonal semiflat projections, which is also semiflat).
\end{proof}

%%%%%%%%%%%%%%%%%%%%%%%%%%%%%%%%%%%%%%%%%
\textcolor{blue}{\large\section{Proofs of Trace Results}}

In this section we prove Theorems \ref{flattraces} and \ref{semiflattraces} stated in the Introduction. To do this we begin with two lemmas.

\medskip

\begin{lem}\label{lemma1}
{\it Let $e, e_1, \dots, e_n$ be Fourier invariant projections in $A_\theta$ such that
\[
t := \vartau(e_1) + \dots + \vartau(e_n) < \vartau(e).
\] 
There are $\sigma$-invariant unitaries $w_1, \dots, w_n$ such that 
\[
w_1e_1w_1^*+ \dots + w_ne_nw_n^*
\]
is a Fourier invariant subprojection of $e$ of trace $t$.
}
\end{lem}
\begin{proof}
Since the order structure on $K_0(A_\theta^\sigma)$ is determined by the canonical trace $\vartau$, the hypothesis implies that there is a projection $Q$ in $M_m(A_\theta^\sigma)$ such that 
\[
[e_1 \oplus \dots \oplus e_n \oplus Q] = [e] = [e \oplus O_r]
\]
in $K_0(A_\theta^\sigma)$, where $r = n + m - 1$ and $O_r$ is the zero $r\times r$ matrix.  By the cancellation property of $A_\theta^\sigma$, 
there is a unitary $W$ in $M_{r+1}(A_\theta^\sigma)$ such that 
\[
W(e_1 \oplus \dots \oplus e_n \oplus Q) W^* = e \oplus O_r.
\]
From this, one obtains a set $f_1,\dots, f_n$ of pairwise orthogonal subprojections of $e$ such that $[f_j] = [e_j]$ for each $j$, which by cancellation again gives unitaries $w_j$ in $A_\theta^\sigma$ such that $f_j = w_je_jw_j^*$.  One therefore gets the projection
\[
w_1e_1w_1^*+ \dots + w_ne_nw_n^*
\]
contained in $e$ and with trace $t$. 
\end{proof}

\medskip

We cite the following lemma from \ccite{SWmathscand}.

\begin{lem}\label{projtraces} (See Theorem 1.6 of \ccite{SWmathscand}.)
{\it Let $\theta$ be irrational, $p/q$ a rational (in reduced form) approximant of $\theta$ such that $0<q|q\theta-p|<1$. Then for each positive integer $k$ such that $k|q\theta-p| < \tfrac14$, there exists a cyclic projection in $A_\theta$ of trace $k|q\theta-p|$.
}
\end{lem}

(Of course, one would then have the corresponding flat projection whose trace is $4k|q\theta-p|$.)

\medskip

\begin{prp}
{\it Let $\theta$ be any irrational number in $(0,1)$.  Then  each number in $(0,1) \cap (4\mathbb Z + 4\mathbb Z\theta)$ is the trace of a $\sigma$-flat projection.
}
\end{prp}
\begin{proof}
Fix $t \in (0,1) \cap (4\mathbb Z + 4\mathbb Z\theta)$. With no loss of generality we can assume $t = 4k(n\theta-m)$ where $k,n\ge1$ and $m\ge0$.  (If $t = 4\ell(r - s\theta)$ where $r,s\ge0$, then $t = 4\ell[s(1-\theta)-(s-r)]$ so that we could obtain a $\sigma$-flat projection in $A_{1-\theta}$ of this trace which maps to $A_\theta$ via a Fourier compatible isomorphism that gives one a $\sigma$-flat projection of trace $t$.)

Since $0 \le \frac{m}n < \theta$, one can choose a pair of consecutive convergents $\frac{p}q, \frac{p'}{q'}$ of $\theta$ such that $0< \frac{m}n < \frac{p}q < \theta < \frac{p'}{q'}$ where $p'q-pq'=1$. We can write
\[
\theta = p'(q\theta-p) + p(p'-q'\theta), \qquad
1 = q'(q\theta-p) + q(p'-q'\theta)
\]
each as sums of positive terms.  Thus
\begin{align*}
t &= 4k[ np'(q\theta-p) + np(p'-q'\theta) - mq'(q\theta-p) - mq(p'-q'\theta) ]
\\
&= 4a(q\theta-p) + 4b(p'-q'\theta)
\end{align*}
where $a = k(np'-mq')$ and $ b = k(np-mq)$ are positive integers.  Now we are in the situation of Lemma \ref{projtraces} which gives us $\sigma$-flat projections $f_1$ and $f_2$ in $A_\theta$ with respective traces $4a(q\theta-p)$ and $4b(p'-q'\theta)$.  Using Lemma \ref{lemma1} there exists a Fourier invariant unitary $w$ such that the orthogonal sum $f_1 + wf_2w^*$ is a flat projection with trace $t$.
\end{proof}

\begin{cor}
{\it Let $\theta$ be irrational.  Then each number in $(0,\tfrac14) \cap (\mathbb Z + \mathbb Z\theta)$ is the trace of a cyclic projection.
}
\end{cor}

\begin{thm}
{\it Let $\theta$ be irrational.  Then each number in $(0,\tfrac12) \cap (\mathbb Z + \mathbb Z\theta)$ is the trace of a semicyclic projection.
}
\end{thm}
\begin{proof}
First, note that each number $2x$ in $(2\mathbb Z + 2\mathbb Z\theta) \cap (0,\tfrac12)$ is the trace of a semicyclic projection. Since $x < \frac14$, the preceding corollary gives a cyclic projection $g$ of trace $x$. The projection $g + \sigma^2(g)$ is then semicyclic of trace $2x$.  

Now fix $t \in (0,\tfrac12) \cap (\mathbb Z + \mathbb Z\theta)$. By density, pick $2x \in (2\mathbb Z + 2\mathbb Z\theta) \cap (0,\tfrac12)$ such that $t < 2x$. By the claim just proved, there exists a semicyclic projection $h$ of trace $2x$. 
Picking a flip-invariant subprojection $e$ of $h$ of trace $t$ (as $t<2x$), one gets  a semicyclic projection $e$ of trace $t$.
\end{proof}

\begin{lem}\label{sumsquares}
Let $\theta\in(0,1)$ be irrational. Then for each pair of integers $m,n$ there exists a Fourier invariant projection of trace $(m^2+n^2)\theta\!\!\mod 1$.
\end{lem}
\begin{proof}
By Theorem 1.1 of \ccite{SWcrelles} for any rational $p/q$ such that $0 < q|q\theta-p| < 1$, there is a Fourier invariant projection of trace $q|q\theta-p|$.  Applying this with $p=0, q=1$, one obtains a Fourier invariant projection of trace $\theta$.  Fix $m,n$ and let
\[
\theta' = (m^2+n^2)\theta\text{ mod 1}.
\]
The unitaries
\[
\widetilde U = e(\tfrac12 mn\theta) V^{-n}U^{m}, \qquad \widetilde V = e(\tfrac12 mn\theta) U^nV^m
\]
are easily checked to satisfy 
\[
\widetilde V \widetilde U = e(\theta') \widetilde U \widetilde V, \qquad 
\sigma(\widetilde U) = \widetilde V^{-1}, \qquad 
\sigma(\widetilde V) = \widetilde U
\]
so that $\sigma$ induces the Fourier transform on the rotation C*-subalgebra $A_{\theta'}$ of $A_\theta$ generated by $\widetilde U$ and $\widetilde V$.  Therefore, by what was just noted, there exists a Fourier invariant projection in $A_{\theta'}$ (hence in $A_\theta$) of trace $\theta'$, as required.
\end{proof}

\begin{thm}\label{TracesFourierInvariant}
Let $\theta$ be irrational in $(0,1)$. Each $t \in (0,1) \cap (\mathbb Z + \mathbb Z\theta)$ is the trace of a Fourier invariant projection.
\end{thm}
\begin{proof}
Write $t = m\theta - n < 1$, and assume, with no loss of generality, that $m > 0$.  By Lagrange's Theorem, write $m = m_1^2 + m_2^2 + m_3^2 + m_4^2$ as a sum of four integer squares.  By Lemma \ref{sumsquares}, there are Fourier invariant projections $e$ of trace $(m_1^2+m_2^2)\theta - n_1 < 1$ and $f$ of trace $(m_3^2+m_4^2)\theta - n_2 < 1$ for some nonnegative integers $n_1, n_2$.  The class $[e] + [f]$ in $K_0(A_\theta^\sigma)$ has positive trace
\begin{equation}\label{sumtraces}
(m_1^2 + m_2^2 + m_3^2 + m_4^2)\theta - n_1 - n_2 \ = \ t  + k \ < \ 2
\end{equation}
where $k=n-n_1-n_2$ is either $0$ or $1$.  If $k = 0$, then since the sum of the traces of $e$ and $f$ is less than 1, Lemma \ref{lemma1} implies that there are unitaries $u,v$ in $A_\theta^\sigma$ such that $ueu^* + vfv^*$ is a Fourier invariant projection of trace $t$.  If $k = 1$, then $[e] + [f] - [1] = [e] - [1-f]$ has positive trace $t$ (from \eqref{sumtraces}), which means that $[e] > [1-f]$ in $K_0(A_\theta^\sigma)$, so that $1-f$ is unitarily equivalent to a subprojection of $e$ by a unitary $w$ in $A_\theta^\sigma$: $w(1-f)w^* \le e$.  One therefore gets the Fourier invariant projection $e - w(1-f)w^*$ of trace $t$.
\end{proof}

%%%%%%%%%%%%%%%%%%%%%%%%%%%%%%%%%%%%%%%%%
\textcolor{blue}{\large\section{Determination of Flat and Semiflat Projections}}

In this section we prove Theorems \ref{identifyflat} and \ref{identifysemiflat}.

\medskip

\begin{thm}
{\it Let $\theta$ be any irrational number, and let $g_1$ and $g_2$ be two cyclic projections in $A_\theta$.
\begin{enumerate}
\item Then $g_1$ and $g_2$ are unitarily equivalent by a unitary in $A_\theta^\sigma$ if and only if $g_1$ and $g_2$ have equal traces.
\item Two flat projections $f_1=\sigma^*(g_1)$ and $f_2 = \sigma^*(g_2)$ are $\sigma$-unitarily equivalent iff they have the same trace iff $g_1$ and $g_2$ $\sigma$-unitarily equivalent.
\end{enumerate}
}
\end{thm}
\begin{proof} We prove (2) first. If $g_1$ and $g_2$ are $\sigma$-unitarily equivalent then clearly so are $f_1$ and $f_2$. So we start with two $\sigma$-unitarily equivalent flat projections $f_1, f_2$. Since the projections $g_1, g_2$ have the same trace, they are Murray von Neumann equivalent in $A_\theta$ (by a theorem of Rieffel). Let $v \in A_\theta$ be a partial isometry such that $vv^* = g_1,\ v^*v=g_2$ and $g_1v=v=vg_2$.  As the projections $g_1$ and $g_2$ are cyclic, one has 
\[
v^*\sigma^j(v) = 0, \qquad v\sigma^j(v^*)=0
\]
for $j=1,2,3$. The first of these follows by replacing $v$ by $g_1v$ and likewise the second by replacing $v$ by $vg_2$. This lends us the Fourier invariant element
\[
w = v + \sigma(v) + \sigma^2(v) + \sigma^3(v)
\]
which acts as a partial isometry between the flat projections
\[
ww^* = f_1, \qquad w^*w = f_2.
\]
Further, one has
\begin{align*}
g_1w &= vv^*[v + \sigma(v) + \sigma^2(v) + \sigma^3(v)] = v, \\
wg_2 &= [v + \sigma(v) + \sigma^2(v) + \sigma^3(v)] v^*v  = v
\end{align*}
where we noted that $vv^*v = v$. These give $f_1w = w = wf_2$.

As the complements $1-f_1$ and $1-f_2$ are also equivalent projections in $A_\theta^\sigma$ (having same $\T_4$'s), there is a $\sigma$-invariant partial isometry $x$ such that $xx^* = 1-f_1$ and $x^*x = 1-f_2$, as well as $(1-f_1)x = x = x(1-f_2)$. Since one has $wx^* = 0 = w^*x$, the element $w + x$ is a $\sigma$-invariant unitary that is easily checked to satisfy
\[
g_1(w+x) = (w+x)g_2.
\]
This proves (2).

To see (1), assume $g_1$ and $g_2$ have equal trace.  Since their corresponding flat projections $f_1 = \sigma^*(g_1), \ f_2 = \sigma^*(g_2)$ also have equal traces, we have $\T_4(f_1) = \T_4(f_2) = (\text{trace}; 0,0; 0,0,0)$ (recalling that the topological invariants $\psi_{\star}$ of flat projections all vanish). By the injectivity of $\T_4$, it follows that $f_1$ and $f_2$ are $\sigma$-unitarily equivalent, and the result follows from (2).
\end{proof}

\begin{rmk}
In view of (2), it also follows that $g_1$ is $\sigma$-unitarily equivalent to $\sigma^j(g_2)$ for any $j$. 
\end{rmk}

\medskip

\begin{thm}
{\it Let $\theta$ be any irrational number, and let $g$ and $h$ be two semicyclic projections in $A_\theta^\Phi$.
\begin{enumerate}
\item Then $g$ and $h$ are $\sigma$-unitarily equivalent iff they are $\Phi$-unitarily equivalent iff  $\T_2(g) = \T_2(h)$.
\item Two semiflat projections $f=g+\sigma(g)$ and $f' = h + \sigma(h)$ are equivalent in $A_\theta^\sigma$ iff $\vartau(g) = \vartau(h)$ and
\[
\phi_{00}(g) = \phi_{00}(h), \qquad \phi_{11}(g) = \phi_{11}(h), \qquad
\phi_{01}(g) + \phi_{10}(g) = \phi_{01}(h) + \phi_{10}(h),
\]
\end{enumerate}
}
\end{thm}
\begin{proof} 
For (1) it is enough to assume that $g$ and $h$ are $\Phi$-unitarily equivalent (since it is already known from \ccite{SWprojmodules} that this is equivalent to $\T_2(g) = \T_2(h)$). Let $u$ be a partial isometry in $A_\theta^\Phi$ such that $uu^* = g, u^*u = h$ (and also $gu=u=uh$). The orthogonalities $g\sigma(g) = h\sigma(h)=0$ give $u^*\sigma(u) = 0 = u\sigma(u^*)$. Letting $w = u + \sigma(u)$, we get
\[
ww^* = g + \sigma(g) =: f, \qquad w^*w = h + \sigma(h) =: f'
\]
(and $fw = w = wf'$) so that $w$ is a $\sigma$-invariant partial isometry. Further, $gw = u = wh$ (as $uu^*u=u$). Since $1-f$ and $1-f'$ are also equivalent projections in $A_\theta^\sigma$ (since they have equal $\T_4$'s), there is a $\sigma$-invariant partial isometry $w'$ such that $w'w'^* = 1-f$ and $w'^*w' = 1-f'$. One then forms the $\sigma$-invariant unitary
\[
W = w + w'
\]
which satisfies $gW = Wh$. 

For statement (2), if $f=g+\sigma(g)$ and $f' = h + \sigma(h)$ are equivalent in $A_\theta^\sigma$ then in view of \eqref{phiandpsi} one has the assertion stated on the $\phi_{ij}$ values (and the trace). The converse follows likewise since the $\T_4$ invariants of $f$ and $f'$ are equal in view of the hypothesis.
\end{proof}

\begin{rmk}
We emphasize that the semicyclic projections $g,h$ in (2) of the preceding theorem are not necessarily equivalent even in the flip fixed point algebra $A_\theta^\Phi$. Indeed, given any $g$ one can consider a semicyclic projection $h$ arising from the equation
\[
\T_2(h) = \T_2(g) + (0; 0, n, -n,0)
\]
for any integer $n$ (where $h$ is obtained in the same manner that lead to Corollary \ref{corh}). Such $h$ gives rise to a semiflat projection $f' = h + \sigma(h)$ equivalent to $f = g + \sigma(g)$ (in $A_\theta^\sigma$), but $h$ is neither equivalent to $g$ nor to $\sigma(g)$. This contrasts with statement (2) of Theorem \ref{identifyflat}.
\end{rmk}

\medskip

\noindent{\bf Acknowledgement.} Too bad NSERC changed its rules which discouraged the author's applying for a research grant, or else it would have been (happily) acknowledged here! Nevertheless, NSERC is thanked for some 20 years of grant support (and the momentum it gave the author).

%%%%%%%%%%%%%%%%%%%%%%%%%%%%%%%%%%%%%%%
%%%%%%%%%%%%%%%%%    REFERENCES      %%%%%%%%%%%
%%%%%%%%%%%%%%%%%    REFERENCES      %%%%%%%%%%%
%%%%%%%%%%%%%%%%%    REFERENCES      %%%%%%%%%%%

\end{document}